\documentclass[12pt,letterpaper,titlepage]{amsart}
\usepackage{amsmath, amssymb, amsthm, amsfonts,amscd,amsaddr,wasysym,enumerate}
\usepackage[
    paper=a4paper,
    portrait=true,
    textwidth=425pt,
    textheight=650pt,
         tmargin=3cm,
    marginratio=1:1
            ]{geometry}

\usepackage[cmtip,matrix,arrow]{xy} \SelectTips{cm}{10}

\theoremstyle{plain}
\newtheorem{theorem}{Theorem}[section]

\newtheorem{cor}[theorem]{Corollary}
\newtheorem{con}[theorem]{Conjecture}
\newtheorem{prop}[theorem]{Proposition}
\newtheorem{lemma}[theorem]{Lemma}

\theoremstyle{definition}

\numberwithin{equation}{section}

\newcommand{\bs}{\backslash}
\newcommand{\Cc}{\mathcal{C}}
\newcommand{\C}{\mathbb{C}}

\newcommand{\Lc}{\mathcal{L}}
\newcommand{\E}{\mathcal{E}}

\newcommand{\Zc}{\mathcal{Z}}

\newcommand{\R}{\mathbb{R}}
\newcommand{\N}{\mathbb{N}}

\newcommand{\Aut}{\operatorname{Aut}}

\newcommand{\Ind}{\operatorname{Ind}}

\newcommand{\Hom}{\operatorname{Hom}}

\newcommand{\Ad}{\operatorname{Ad}}

\newcommand{\ad}{\operatorname{ad}}

\newcommand{\id}{\operatorname{id}}
\newcommand{\interior}{\operatorname{int}}

\newcommand{\Spec}{\operatorname{spec}}

\newcommand{\rank}{\operatorname{rank}}

\newcommand{\re}{\operatorname{Re}}

\newcommand{\algebraicgroup}[1]{{\underline{#1}}}

\newcommand{\uG}{\algebraicgroup{G}}
\newcommand{\uH}{\algebraicgroup{H}}

\def\af{\mathfrak{a}}

\def\gf{\mathfrak{g}}

\def\cf{\mathfrak{c}}

\def\hf{\mathfrak{h}}
\def\jf{\mathfrak{j}}
\def\kf{\mathfrak{k}}

\def\mf{\mathfrak{m}}
\def\nf{\mathfrak{n}}

\def\sf{\mathfrak{s}}

\def\tf{\mathfrak{t}}
\def\uf{\mathfrak{u}}

\def\la{\langle}
\def\ra{\rangle}
\def\1{{\bf1}}

\def\U{\mathcal{U}}

\def\Cc{\mathcal{C}}

\def\P{\mathcal{P}}

\def\oline{\overline}

\title[Ellipticity]
{Ellipticity and discrete series
\smallbreak
{\tiny {\it D\MakeLowercase{edicated to} J\MakeLowercase{oseph} B\MakeLowercase{ernstein for his appreciation of soft methods}}}}

\begin{document}

\begin{abstract} We explain by elementary means why the existence of a discrete series representation
of a real reductive group $G$ implies the existence of a compact Cartan subgroup of $G$. The presented approach has the potential to generalize to real spherical spaces.
\end{abstract}

\author[Kr\"otz]{Bernhard Kr\"{o}tz}
\email{bkroetz@gmx.de}
\address{Universit\"at Paderborn, Institut f\"ur Mathematik\\ Warburger Stra\ss e 100,
33098 Paderborn}

\author[Kuit]{Job J. Kuit}
\email{j.j.kuit@gmail.com}
\address{Universit\"at Paderborn, Institut f\"ur Mathematik\\Warburger Stra\ss e 100,
33098 Paderborn}

\author[Opdam]{Eric M. Opdam}
\email{e.m.opdam@uva.nl}
\address{University of Amsterdam, Korteweg-de Vries Institute for Mathematics\\P.O. Box 94248, 1090 GE Amsterdam}

\author[Schlichtkrull]{Henrik Schlichtkrull}
\email{schlicht@math.ku.dk}
\address{University of Copenhagen, Department of Mathematics\\Universitetsparken 5,
2100 Copenhagen \O}

\maketitle

\section{Introduction}

Let $\uG$ be a connected reductive algebraic group defined over $\R$ and $G:=\uG(\R)$ its group of real points.
In this article we give an elementary proof that Harish-Chandra's compact Cartan subgroup
condition is necessary for $G$ to have discrete series. To explain the background,
we first describe the problem in the more general context of real spherical spaces.

\subsection{Real spherical spaces}
Let $\uH\subset\uG$ be an algebraic subgroup defined over $\R$ and $H=\uH(\R)$.
A suitable framework for harmonic analysis on $Z:=G/H$ is obtained by the request
that $Z$ is real spherical, i.e., there exists an open orbit on $Z$  for the natural action
of a minimal parabolic subgroup $P$ of $G$.

\par Our interest is to obtain a geometric criterion
for the existence of discrete series on a  unimodular real spherical space
$Z$. We recall that by definition
the discrete series for $Z$ consists of the irreducible subrepresentations of the
regular representation of $G$ on $L^2(Z)$. The following condition
for its existence was conjectured in  \cite[(1.2)]{KKOS}:

\begin{con} \label{con intro} Let $Z$ be a unimodular real spherical space.
A necessary and sufficient condition for the existence of a discrete series representation for $Z$
is that the interior of $(\hf^\perp)_{\rm ell}$ in $\hf^{\perp}$ is non-empty.
\end{con}

Let us explain the notation. Let $\gf, \hf$ be the Lie algebras of $G$ and $H$. Then
$\hf^\perp \simeq(\gf/ \hf)^*$ is the cotangent space $T_{z_0}^*Z$ at the base point $z_0=H\in Z$, and the index
`ell' stands for elliptic elements.

The sufficiency of the condition has been established in \cite{DKKS}. We recall the result:

\begin{theorem}\label{th intro} Let $Z$ be a unimodular real spherical space.
If the interior of $(\hf^\perp)_{\rm ell}$  in $\hf^{\perp}$ is non-empty,
then there exist infinitely many representations in the discrete series for $Z$.
\end{theorem}

A central tool in the proof of this theorem is a property
of the infinitesimal characters of discrete series representations for $Z$, derived  in \cite{KKOS}. The same property is
crucial for our approach to necessity. Some notation is needed
in order to describe it.

Let $G=KAN$ be an Iwasawa decomposition
for $G$ and $P=MAN$ the associated minimal parabolic subgroup, with $M=Z_K(A)$ the centralizer of $A$ in $K$.
Let $\tf\subset \mf$ be a maximal torus. Then $\cf = \af +\tf$ is a maximally split Cartan subalgebra of $G$, unique up to
conjugation. With $\cf_\R = \af + i \tf$ we obtain a real form
of $\cf_\C$ which is characterized by the property that all roots $\gamma \in \Sigma_\cf = \Sigma(\gf_\C, \cf_\C) \subset \cf_\C^*$ are real valued on $\cf_\R$. Let $V$ be the Harish-Chandra module of a
discrete series representation for $Z$, and let its  infinitesimal character be denoted
$\chi_V \in \Hom_{\rm alg} (\Zc(\gf), \C)$. Using the Harish-Chandra isomorphism
we identify $\chi_V$ with a $W_\cf$-orbit $[\Lambda_V]=W_\cf\cdot \Lambda_V \in\cf_\C^*/ W_\cf$,
where $W_{\cf}$ is the big Weyl group, i.e. the Weyl group of the root system  $\Sigma_{\cf}$ with respect to the Cartan subalgebra $\cf$.

The mentioned result of \cite{KKOS} asserts that there exists an explicit
$W_\cf$-invariant rational lattice $\Lc$, such that
\begin{equation} \label{integral} [\Lambda_V]\subset \Lc \subset \cf_\R^* \end{equation}
for all discrete series representations $V$ of $Z$.  Let us emphasize  in particular that the
parameters $\Lambda_V$ of the discrete series are {\it real}, as the lattice $\Lc$ lies in the real form
$\cf_\R^*$.

The purpose of this article is to explore whether this property of the
infinitesimal character can be used to establish the conjectured necessity of the condition.
To be more precise,
we show that this is the case for the group, regarded as a spherical space. We believe the approach
has the potential to generalize to all real spherical spaces.

\subsection{The group case} In the remainder of this article we consider the group case.
The group $G$ is a real spherical space when looked upon as a geometric
object under its both-sided symmetries of $G \times G$.
 Specialized to this case the conjecture is Harish-Chandra's  beautiful geometric criterion for the existence of discrete series representations for $G$, which results from his deep study of discrete series \cite{HC1, HC2}.

\begin{theorem}\label{thm H-C}{\rm (Harish-Chandra, \cite[Theorem 13]{HC2})}
A necessary and sufficient condition for
$G$ to admit discrete series is that it has a compact Cartan subgroup.
\end{theorem}

\par As mentioned, we provide an elementary proof of the necessity,
based on the property \eqref{integral} for $G$. In the case at hand the proof of this property
is also elementary, as explained in the introduction to \cite{KKOS}.

Let us describe the argument. Let $\sigma$ be the conjugation on $\gf_{\C}$ with respect to $\gf$. We call an element $\Lambda\in \cf_\C^*$ {\it strongly regular} provided that the stabilizer
of $\Lambda$ in the extended Weyl group $W_{\cf, \text{ext}}:=\la W_\cf, -\sigma\ra \subset
\Aut(\Sigma_\cf)$ is trivial.  We show that the existence of a unitary representation with a strongly regular real infinitesimal character implies the existence of a compact
Cartan subgroup, see Corollary \ref{cor sr implies elliptic}.
Knowing that infinitesimal characters of discrete series
are real, the existence of a discrete series representation with strongly regular
infinitesimal character therefore requires the existence of a compact Cartan subgroup.
Finally, we complete the proof by using the
Zuckerman translation principle \cite{Z} to produce from any representation of
the discrete series a discrete series representation with strongly regular infinitesimal
character, see Corollary \ref{cor HC-criterion}. The tools used for this belong to general
representation theory of Harish-Chandra modules. Beyond
the characterization of square integrability in terms of the leading exponents of
asymptotic expansions, the only property of discrete series used at this stage
is the existence of the lattice $\Lc$ satisfying \eqref{integral}.

\medbreak

{\it Acknowledgement:} We thank Joseph Bernstein and the referee for valuable comments.

\section{Notation}

Throughout this article we let $G$ be the open connected subgroup
of $\uG(\R)$ where $\uG$ is a connected reductive group defined over $\R$. We write $G_{\C}$ for the connected group $\uG(\C)$.
As usual we denote the Lie algebra of $G$ by $\gf$ and keep this terminology for subgroups of $G$, i.e., if
$H\subset G$ is a subgroup, then we denote by $\hf$ its Lie algebra. If $\hf$ is a Lie algebra, then we write $\hf_{\C}$ for the complexification of $\hf$.

\par Fix a Cartan involution $\theta$ of $G$ and denote by $K=G^\theta$ the corresponding maximal compact subgroup.
The Lie algebra automorphism of $\gf$  induced by $\theta$, and its linear extension to $\gf_{\C}$, will be denoted by $\theta$ as well. We write
$\gf=\kf +\sf$ for the associated Cartan decomposition.  We fix a maximal abelian subspace
$\af \subset \sf$ and write $A=\exp(\af)$. Further we let $M=Z_K(A)$ and select with
$\tf\subset \mf$ a maximal torus.  We write $T$ for the Cartan subgroup $Z_{M}(\tf)$ of $M$.

We denote by $\sigma: \gf_\C\to \gf_\C$ the complex conjugation with respect to the real form
$\gf$, and let $U:=K \exp(i\sf)$ be the $\theta$-stable maximal compact subgroup of $G_{\C}$, which is obtained as the fixed point subgroup of the antilinear extension
$\theta \circ \sigma$ of the Cartan involution $\theta$ to $G_\C$.

\par We extend $\af$ by  $\tf$ to a Cartan
subalgebra $\cf:=\af +\tf$ of $\gf$, and use the symbol $\sigma$ also for the restriction of $\sigma$ to $\cf_\C$. We write $\Sigma_\cf=\Sigma(\gf_\C, \cf_\C)$ for the
corresponding root system and $\Sigma_\af:=\Sigma_\cf|_{\af}\bs \{0\}$ for the corresponding restricted root
system.
Further we set $\cf_\R:= \af + i\tf$.
Note that $\Sigma_\cf \subset \cf_\R^*$, that $\sigma$ preserves $\Sigma_\cf$ and $\cf_\R$ and that $\sigma\big|_{\cf_{\R}}=-\theta\big|_{\cf_{\R}}$.
We write $C_{\C}$ for the maximal torus of $G_{\C}$ with Lie algebra $\cf_{\C}$. As $G_{\C}$ is a connected algebraic reductive group, the torus $C_{\C}$ is connected. We further define $C:=G\cap C_{\C}$ and  $C_{U}:=C_{\C}\cap U$. Note that $C=TA$ and $C_{U}=T\exp(i\af)$.

\par Let us denote by $W_\cf$ the Weyl group of the root system $\Sigma_\cf$ and likewise we
denote by $W_\af$ the Weyl group of the restricted root system $\Sigma_\af$.
With respect to $\Sigma_\af$ we have the
restricted root space decomposition
$$ \gf = \af \oplus \mf \oplus\bigoplus_{\alpha\in \Sigma_\af} \gf^\alpha\, .$$

\par In the sequel we fix with $\Sigma_\af^+\subset \Sigma_\af$ a positive system. We then let $\Sigma_\cf^+\subset \Sigma_\cf$ be any  positive system which is compatible with $\Sigma_\af^+$, i.e.,
$\Sigma_\af^+=\Sigma_\cf^+|_\af \bs\{0\}$.

\par The positive system $\Sigma_\af^+$ defines a maximal nilpotent subalgebra $\nf=\bigoplus_{\alpha\in \Sigma_\af^+}
\gf^\alpha$. Put $N=\exp \nf$ and note that $P=MAN\subset G$ defines a minimal parabolic subgroup of $G$. We write $\overline{P}$ and $\overline{\nf}$ for $\theta P$ and $\theta\nf$, respectively.

\section{Reading of the existence of maximal compact Cartan subgroup from the infinitesimal character}

As usual we write $\Zc(\gf)$ for the center of the universal enveloping algebra
$\U(\gf)$ of $\gf_\C$.  Recall that according to Harish-Chandra the characters
$\chi$ of $\Zc(\gf)$ are parametrized
by $\cf_\C^*/ W_\cf$ as follows.
For any positive system $S$ of $\Sigma_{\cf}$ we set
$$
\uf_{S}
:=\bigoplus_{\alpha\in S}\gf_{\C,\alpha},
$$
and write $\rho_{S}$  for half the trace of $\ad(\cf)$ on $\uf_{S}$.
Using the Poincaré-Birkhoff-Witt theorem we may decompose an element $Z\in\Zc(\gf)$ as
\begin{equation}\label{eq Z decomp}
Z
\in C_{S}+\uf_{-S}\,\U(\gf)\uf_{S}
\end{equation}
with $C_{S}\in \U(\cf)$, see the proof of  \cite[Lemma 8.17]{K}. The element $[\Lambda]\in\cf_\C^*/ W_\cf$
parametrizing $\chi$ is then given by
\begin{equation}\label{eq H-C iso}
\chi(Z)
=(\Lambda-\rho_{S})(C_{S})
\end{equation}
and does not depend on the choice of $S$.

\par Every irreducible Harish-Chandra module $V$ admits an infinitesimal character
$\chi_V:\Zc(\gf) \to \C$ which then corresponds to a $W_\cf$-orbit
$$
[\Lambda_V]
:=W_\cf\cdot \Lambda_V
$$
for some $\Lambda_V\in \cf_\C^*$.
The following lemma is standard. For convenience we include its short proof.

\begin{lemma} \label{Lemma infinitesimal}
Let $V$ be an irreducible Harish-Chandra module.  The following hold.
 \begin{enumerate}
 \item\label{Lemma infinitesimal - item 1} $[\Lambda_{\tilde{V}}]=[-\Lambda_{V}]$, where $\tilde{V}$ is the contragredient of $V$.
 \item\label{Lemma infinitesimal - item 2} If $V$ is unitarizable, then $[\Lambda_V]=[-\sigma\Lambda_V]$.
\end{enumerate}
\end{lemma}

\begin{proof}
Let $Z\mapsto Z^\vee$ denote the principal anti-automorphism of $\U(\gf)$.
Then $\chi_{\tilde V}(Z)=\chi_V(Z^\vee)$ for $Z\in\Zc(\gf)$. Let $S$ be any positive system of $\Sigma_{\cf}$. Let $Z\in\Zc(\gf)$ and let $C_{S}\in \U(\cf)$ be as in (\ref{eq Z decomp}).
By \eqref{eq H-C iso}
$$
\chi_{\tilde{V}}(Z)
=(\Lambda_{\tilde{V}}-\rho_{S})(C_{S}).
$$
As
$$
Z^{\vee}
\in C_{S}^{\vee}+\uf_{S}\,\U(\gf)\uf_{-S},
$$
we have
$$
\chi_{V}(Z^{\vee})
=(\Lambda_{V}-\rho_{-S})(C_{S}^{\vee})
=(-\Lambda_{V}-\rho_{S})(C_{S}).
$$
This proves (\ref{Lemma infinitesimal - item 1}).

The conjugate representation $\oline V$  of $V$ has infinitesimal character $[\Lambda_{\oline V}]= [\sigma(\Lambda_V)]$. If $V$ is unitarizable, then the representation is isomorphic to its conjugate dual, hence assertion (\ref{Lemma infinitesimal - item 2}).
\end{proof}

We recall that an element $\lambda\in \cf_\C^*$ is {\it regular} provided that the stabilizer of $\lambda$ in
$W_\cf$ is trivial.  Notice that the complex conjugation $\sigma$ and $-\id$ induce
automorphisms of $\Sigma_\cf$, i.e., they determine elements of $\Aut(\Sigma_\cf)$.
In particular
$-\sigma\in \Aut(\Sigma_\cf)$. We define the extended Weyl group of $W_\cf$ as the following subgroup of $\Aut(\Sigma_\cf)$:
$$W_{\cf, \rm ext}:=\la W_\cf, -\sigma\ra_{\text{group}} \subset \Aut(\Sigma_\cf)\, .$$
Furthermore $\lambda \in \cf_\C^*$ is called {\it strongly regular} in case
the stabilizer in $W_{\cf , \text{ext}}$ is trivial.

According to Harish-Chandra (see \cite[Theorem 16]{HC2}) the infinitesimal characters of representations of the discrete series $V$
of $G$ are real, i.e., $\Lambda_{V}\in \cf_\R^*/W_\cf$.  A simplified proof of this
fact was recently given in the more general context of real spherical spaces, see \cite[Theorem 1.1]{KKOS}.

\begin{prop} \label{prop sigma-Weyl}Assume that there exists a representation $V$ of the discrete series for $G$ with
infinitesimal character $[\Lambda]\in \cf_\C^*/ W_{\cf}$.  Then the following assertions hold:
\begin{enumerate}
\item $\Lambda\in \cf_{\R}^{*}$ and there exists an element
$w\in W_\cf$ such that $w \cdot \Lambda = -\sigma(\Lambda)$.
\item If in addition $\Lambda$ is strongly regular, then there exists an element $w\in W_\cf$ such that
$w=-\sigma$ on $\cf_\R^*$.  In particular, $-\sigma|_{\cf_\R^*}\in W_\cf\subset \Aut(\cf_\R^*)$.
\end{enumerate}
\end{prop}

\begin{proof} As mentioned above,  $\Lambda\in\cf_{\R}^{*}$. Since representations of the discrete series are also unitarizable,
Lemma \ref{Lemma infinitesimal}  gives $[-\sigma \Lambda]= [\Lambda]$.  This shows the
first assertion and the second is a consequence thereof.
\end{proof}

We recall that $W_\cf= N_{G_{\C}}(\cf_{\C})/C$
and $W_\af= N_K(\af)/M$.
We denote by $W_\cf^\theta$ the subgroup of $W_{\cf}$ consisting of the elements which commute with $\theta$, and recall the exact sequence

\begin{equation} \label{exact}  1\to W_\mf \to W_\cf^\theta \to W_\af\to 1\end{equation}
where $W_\mf$ is the Weyl group of the root system $\Sigma_\mf:= \Sigma(\mf_\C,  \tf_\C)$, which can be realized as $N_M(\tf)/T$.

\begin{lemma}\label{lemma tau diagonal -> inner}
Let $\tau$ be an automorphism of $\gf_{\C}$ and $J_{\C}$ a Cartan subgroup of $G_{\C}$. If $\tau$ acts trivially on $\jf_{\C}$, then there exists a $t\in J_{\C}$ so that $\tau=\Ad(t)$.
\end{lemma}

\begin{proof}
Since $\tau$ acts trivially on $\jf_{\C}$, it preserves all root spaces $\gf_\C^\gamma$, $\gamma \in \Sigma_\jf$.
Hence there exists for all $\gamma\in \Sigma_\jf$ numbers $c_\gamma\in\C$ such that
$\tau|_{\gf_\C^\gamma}= c_\gamma\cdot  \id_{\gf_\C^\gamma}$.
Let now $t\in  J_{\C}$ be such that $\Ad(t)$ coincides with
$\tau$ on all simple root spaces $\gf_\C^\gamma$, $\gamma\in \Pi_\jf$.
Now $\phi:=\Ad(t)^{-1}\circ \tau$
is an automorphism of $\gf_\C$ which acts trivially on ${\mathfrak b}_\C=\jf_\C +\bigoplus_{\gamma\in\Sigma_\jf^+} \gf_\C^\gamma$
and leaves all other $\gf_\C^{-\gamma}$, $\gamma\in \Sigma_{\jf}^+$, invariant.
In fact, $\phi$ acts trivially on all negative root spaces.  To see this, let $\gamma\in \Sigma_{\jf}^+$ and
$0\neq E_\gamma\in \gf_\C^{\gamma}$ and $0\neq F_\gamma\in \gf_\C^{-\gamma}$.
Then $0\neq [E_\gamma, F_\gamma]\in \jf_\C$. As $\phi$ acts trivially on $\jf_{\C}$ we have
$$
[E_{\gamma},F_{\gamma}]
=\phi[E_{\gamma},F_{\gamma}]
=[E_{\gamma},\phi F_{\gamma}],
$$
and hence $\phi F_{\gamma}=F_{\gamma}$.  It follows that $\tau=\Ad(t)$.
\end{proof}

\begin{prop}\label{prop equivalences to theta inner}
 The following assertions are equivalent:
\begin{enumerate}
\item \label{eins}  $-\sigma|_{\cf_\R} \in W_\cf$.
\item \label{zwei} $\theta|_{\cf_\C} \in W_\cf$.
\item \label{three} $\theta$ is an inner automorphism of $\gf_{\C}$.
\item\label{four} There exists a $g\in U$ such that $\theta=\Ad(g)$ as an automorphism of $\gf_{\C}$.
\end{enumerate}
\end{prop}

\begin{proof} Since $-\sigma$ and $\theta$ coincide on $\cf_\R$, the equivalence
of \eqref{eins} and \eqref{zwei} is clear.

\par Suppose now that \eqref{zwei} holds.
Since $\theta\big|_{\af}\in W_{\af}$  there exists a $k\in N_{K}(\af)$ so that $\theta\big|_{\af}=\Ad(k)\big|_{\af}$. Since $N_{K}(\af)\subseteq N_{K}(M)$, the restriction of $\Ad(k)^{-1}\theta$ to $\cf_{\R}$ defines an element of $W_{\cf}$ whose restriction to $\af$ is trivial. In view of \eqref{exact} $\Ad(k)^{-1}\theta$ defines an element of $W_{\mf}$, and thus there exists an $m\in M$ so that $ \Ad(k)^{-1}\theta\big|_{i\tf}=\Ad(m)\big|_{i \tf}$. Now $\Ad(km)$ and $\theta$ coincide on $\cf_{\R}$. Let $w=km$.

\par Let $\tau=\theta\circ \Ad(w)$. Since $\tau$ is an automorphism of $\gf_{\C}$ with $\tau|_{\cf_\C}= \id_{\cf_\C}$, it follows from Lemma \ref{lemma tau diagonal -> inner} that there exists a $t\in C_{\C}$ so that $\tau=\Ad(t)$.
Since $\theta$ commutes with $\Ad(w)$ (as $w$ in $K$) we have $\tau^2 \in \Ad(K)$. Hence $\langle\tau\rangle=\langle\tau^{2}\rangle\cup\tau\langle\tau^{2}\rangle$ is a relatively compact subgroup of $\Ad(C_\C)$. Consequently we see that $t$ can in fact be chosen in $C_{U}$.  It follows $\theta = \Ad(t w^{-1})$ with $g:= t w^{-1}\in U$. This proves (\ref{four}).

The implication of (\ref{three}) from (\ref{four}) is trivial.

Finally, if (\ref{three}) holds, then there exists a $g\in G_{\C}$ so that $\theta=\Ad(g)$. Since $\theta$ preserves the Cartan subalgebra $\cf_{\R}$, we have $g\in N_{G_{\C}}(\cf_{\R})$. Therefore, $\theta|_{\cf_\C}=\Ad(g)|_{\cf_\C}\in W_{\cf}$.  This proves (\ref{zwei}).
\end{proof}

The following statement can also be found in \cite[Lemma 1.6]{A}.

\begin{cor}\label{cor theta inner -> cpt Cartan}
The Cartan involution $\theta$ is an inner automorphism of $\gf_{\C}$ if and only if $\kf\subset \gf$ is a reductive subalgebra
of maximal rank. In that case $\gf$ admits a compact Cartan subalgebra.
\end{cor}

\begin{proof}
Assume that $\theta$ is an inner automorphism of $\gf_{\C}$.  By Proposition \ref{prop equivalences to theta inner} there exists a $g\in U$ so that $\Ad(g)=\theta$. As $g$ is semi-simple, the group
$K_\C:=G_{\C}^{\theta}$ is equal to $Z_{G_\C}(g)$. The centralizer of a semi-simple element contains a maximal torus of $G_\C$, and therefore,
$\rank K_\C = \rank G_\C$.

If $\kf$ is reductive of maximal rank, then there exists a Cartan subalgebra $\hf$ of $\gf$ in $\kf$. The Cartan involution $\theta$ acts trivially on $\hf$. Now Lemma \ref{lemma tau diagonal -> inner} is applicable to $\tau=\theta$  and $\jf_{\C}=\hf_{\C}$. It follows that $\theta$ is inner.
\end{proof}

\begin{cor}\label{cor sr implies elliptic}  Suppose that there exists a representation of the discrete series for $G$ with
strongly regular infinitesimal character. Then $G$ admits a compact Cartan subgroup.
\end{cor}

\begin{proof}
The assertion follows from Propositions \ref{prop sigma-Weyl} and \ref{prop equivalences to theta inner} and Corollary \ref{cor theta inner -> cpt Cartan}.
\end{proof}

\section{Power series expansion}

In this section we summarize a few basic facts regarding the power series expansions of
the matrix coefficients
of an irreducible Harish-Chandra module $V$. We denote the dual Harish-Chandra module of
$V$ by $\widetilde V$.  Recall that $\widetilde V$ is given by the $K$-finite vectors in the algebraic dual
$V^*$ of $V$.
As before we identify the infinitesimal character of $V$ with an $W_\cf$-orbit $[\Lambda_V]= W_\cf\cdot \Lambda_V \subset \cf_\C^*$.

\par Let us denote by $\af^{++}$ the positive Weyl chamber in $\af$ with respect to $\Sigma_\af^+$ and denote by
$\af^+$ the closure of $\af^{++}$. Likewise we set $A^{++}=\exp(\af^{++})$ and $A^+=\exp(\af^+)$.
As usual we denote by $\rho=\frac{1}{2} \sum_{\alpha\in \Sigma_\af^+} (\dim \gf^\alpha)\alpha\in \af^*$ the Weyl half sum.

Now given an irreducible  Harish-Chandra module $V$ each $K$-bi-finite matrix coefficient
$$G\ni g\mapsto  m_{v,\tilde v} (g):=\la \pi(g) v, \tilde v\ra$$
 for $v\in V$ and $\tilde v\in \widetilde V$ admits a power series expansion
on $A^{++}$, see \cite[Ch. VIII]{K}.  To be precise, we have
$$
m_{v,\tilde v}(a)
=\sum_{\xi\in [\Lambda_V]|_\af - \N_0[\Sigma_\af^+]}
p_{v, \tilde v}^{\xi}(\log a) \,a^{\xi- \rho}
\qquad (a\in A^{++},  v\in V, \tilde v\in \widetilde V)
$$
with unique polynomials $p_{v, \tilde v}^{\xi}$ on $\af$
which are of bounded degree and depend bilinearly on the pair $v, \tilde v$.
In case $V$ belongs to the  discrete series  only those elements $\xi$ contribute
for which $\re \xi|_{\af^+}$ is negative, i.e., $\re  \xi(X)<0$ for all $X\in \af^+\bs\{0\}$.

By definition, an element
$\xi\in [\Lambda_V]|_\af - \N_0[\Sigma_\af^+]$
is called an {\it exponent} of $V$ if $p_{v, \tilde v}^{\xi}\neq 0$
for some $v,\tilde v$. The maximal elements in the set of exponents with respect to
the ordering given by $\xi_1\succeq\xi_2$ if $\xi_1-\xi_2\in\N_0[\Sigma_\af^+]$
are called the {\it leading exponents}.
We denote by $\E_V\subset \af_\C^*$ the set of leading exponents and note that
by \cite[Theorem 8.33]{K} we have $\E_V\subseteq [\Lambda_V]|_{\af}$. Then
$$m_{v,\tilde v}(a) = \sum_{\xi\in \E_V- \N_0[\Sigma_\af^+]}
p_{v, \tilde v}^{\xi}(\log a) a^{\xi-\rho}
\qquad (a\in A^{++})\, .$$

The coefficients $p_{v,\tilde v}^\lambda$ for $\lambda\in\E_V$ determine the
principal asymptotics of the matrix coefficient in the sense that

$$
m_{v,\tilde v}(a)
= \sum_{\lambda\in \E_V}  p_{v,\tilde v}^\lambda(\log a) a^{\lambda- \rho}
+\hbox{lower order terms}\qquad (a\in A^{++})\, .
$$

The condition that $V$ belongs to the discrete series can be read off by its set of leading exponents. Let
$$\Cc:=(\af^+)^\star:=\{ \lambda\in \af^*\mid \lambda(X)\geq 0, \ X\in \af^+\}= \sum_{\alpha\in \Sigma^+} \R_{\geq 0}\alpha$$
be the dual Weyl chamber.  By \cite[Theorem 8.48]{K}  $V$ belongs to the discrete series if and only if it satisfies the condition
\begin{equation}\label{L^2-condition}
\re \E_V
\subset - \interior \Cc.
\end{equation}

\begin{lemma} \label{lemma tensor}
Let  $F=F_\mu$ be a finite dimensional representation of $G$ with highest weight $\mu$ with respect to $\Sigma_\cf^+$ and let $V$ be a Harish-Chandra module of the discrete series. The following are equivalent:
\begin{enumerate}
\item $\re\mu|_{\af}+\re\E_V  \subset - \interior \Cc$.
\item All matrix coefficients of $V\otimes F_\mu$ are contained in $L^2(G)$.
\end{enumerate}
\end{lemma}

\begin{proof}
If $v\otimes f\in V\otimes F_{\mu}$ and $\tilde{v}\otimes \tilde{f}\in\widetilde{V}\otimes F_{\mu}^{*}$, then
\begin{equation}\label{eq matrix coeff}
m_{v\otimes f,\tilde{v}\otimes\tilde{f}}
=m_{v,\tilde{v}}\,m_{f,\tilde{f}}.
\end{equation}
The assertion (1) $\Rightarrow$ (2) now follows from  \eqref{L^2-condition} as $\Spec_\af F_\mu \subset \mu|_\af - \N_0[\Sigma_\af^+]\subset \mu|_\af-\Cc$. The other implication follows immediately from (\ref{eq matrix coeff}) with suitable choices of $f$ and $\tilde{f}$.
\end{proof}

\section{Application of the translation principle}
For a Harish-Chandra module $V$ we denote by $H_0(\oline\nf, V)=V/\oline\nf V$ the finite dimensional
$\oline \nf$-homology of degree $0$, and recall that the covariant functor $H_0(\oline\nf, \,\cdot\,)$ is right exact. 
Notice that $H_0(\oline\nf, V)$ is a module for $MA$.
By the Harish-Chandra homomorphism
we have $\Zc(\mf)\simeq \U(\tf)^{W_\mf}$. Moreover we note $\Zc(\af +\mf)= \U(\af) \otimes \Zc(\mf)$. Therefore we can consider the spectrum of a finite dimensional $\Zc(\af+\mf)$-module as a $W_\mf$-invariant subset of
$\cf_\C^*$. In addition we consider $\rho$ as a $W_{\mf}$-invariant element of $\cf_{\C}^{*}$ by extending it trivially on $\tf$.

\begin{lemma}\label{Lemma Hp} Let $V$ be an irreducible Harish-Chandra module with infinitesimal character
$[\Lambda]$. Then the following assertions hold:
\begin{enumerate}
\item  \label{spec Hp} $\Spec_{\Zc(\af+\mf)} H_0(\oline \nf, V)\subset -\rho + [\Lambda]$.
\item\label{spec Hp3} $ \Spec_\af H_0(\oline\nf, V) \subset -\rho +\E_V - \N_0[\Sigma_\af^+]$.
\end{enumerate}
\end{lemma}

\begin{proof} For \eqref{spec Hp} see \cite[Cor. 3.32]{HS2}.
For the inclusion in \eqref{spec Hp3}, let $\lambda\in \Spec_\af H_0(\oline\nf, V)$. Recall that it follows from Casselman's version of Frobenius reciprocity that elements $\lambda\in \Spec_{\af} H_0(\oline \nf, V)$
correspond to embeddings of $V$ into a minimal principal series representation $\Ind_{\overline{P}}^{G}\big(\sigma\otimes(\lambda+\rho)\big)$ (see \cite{Casselman} or \cite[Theorem 4.9]{HS2}).
Without loss of generality, we may assume that $V\subset\Ind_{\overline{P}}^{G}\big(\sigma\otimes(\lambda+\rho)\big)$. As in the derivation of  \cite[(1.4)]{KKOS} one sees that $\lambda+\rho$ occurs as an exponent of $V$, and hence is contained in $\E_V -\N_0[\Sigma_\af^+]$.
\end{proof}

For the rest of this section we let $V$ be a Harish-Chandra module of the discrete series with infinitesimal character
$[\Lambda]= W_\cf\cdot \Lambda \in \cf_\C^*/ W_\cf$.  We set

$$
[\Lambda]^+
:=\{\nu\in [\Lambda]\mid \re\nu|_{\af}\in-\interior\Cc\}
=\{ \nu\in [\Lambda]\mid  \re \nu|_{\af^+\bs \{0\}}<0\}\, .
$$

\begin{lemma}
Let $V$ be a Harish-Chandra module of the discrete series with infinitesimal character $[\Lambda]$.
Then
\begin{equation} \label{Spec Hp3}
\Spec_{\Zc(\af+\mf)} H_0(\oline \nf, V)\subset -\rho + [\Lambda]^+ \, .
\end{equation}
\end{lemma}

\begin{proof} Immediate from Lemma \ref{Lemma Hp}  and (\ref{L^2-condition}).
\end{proof}

We pick the representative $\Lambda\in[\Lambda]$
such that $\lambda:=\Lambda|_\af\in \E_V$.
In view of \cite{KKOS}, Theorem 1.1 and Remark 1.2(3), there exists an $N\in \N$, independent of the discrete series representation $V$, so that
$N\Lambda$ is integral. We select such an $N$ and set $\mu_0:=N\Lambda$. Let
$\mu$  be the unique dominant integral element in $W_\cf \cdot \mu_0$ and let $F_\mu$ be
the corresponding finite dimensional representation of $G$ with highest weight $\mu\in \cf_\R^*$.

We are interested in the $\Zc(\gf)$-isotypical decomposition of $V\otimes F_\mu$. Let $\chi_{\Lambda +\mu_{0}}:\Zc(\gf)\to\C$ be the character corresponding to $[\Lambda +\mu_{0}]$.
According to Zuckerman \cite[Theorem 1.2 (1)]{Z}  the element $[\Lambda+\mu_0]$ appears in $\Spec_{\Zc(\gf)} (V\otimes F_\mu)$
and thus the corresponding isotypical component
\begin{equation} \label{def U}
W
:=\{ v \in V\otimes F_\mu\mid  (\exists k\in \N) (\forall z\in \Zc(\gf))\ (z-\chi_{\Lambda +\mu_{0}} (z))^k \cdot v =0  \}
\end{equation}
is non-zero. Let $J\subset W$ be a maximal submodule and set $U:=W/J$. Then $U$ is an irreducible Harish-Chandra
module with infinitesimal character
$$[\Lambda_U]=[\Lambda+\mu_0] = [(N+1)\Lambda]\, .$$

\begin{lemma}\label{lemma Spec HP}  For any finite dimensional representation $F$ we have
$$ \Spec _{\Zc(\af+\mf)} H_0(\oline\nf, V\otimes F)\subset - \rho + [\Lambda]^+ + \Spec_{\Zc(\af+\mf)} F\, .$$
\end{lemma}

\begin{proof} Filter $F$ as $\oline P$-module
as
$$ F_0=\{0\} \subsetneq F_1 \subsetneq\ldots \subsetneq F_n =F$$
such that $F_k/ F_{k-1}$ is an irreducible $\oline P$-module for each $1\leq k \leq n$. In particular,
each $F_k/ F_{k-1}$ is a trivial $\oline\nf$-module and thus
$H_0( \oline\nf, V \otimes F_k/ F_{k-1}) = H_0(\oline\nf, V) \otimes F_k/ F_{k-1}$ as
$MA$-modules.

We apply now $H_0$ to the exact sequence of $MA$-modules
$$
0\to V\otimes F_{k-1}\to V\otimes F_{k}\to V\otimes F_{k}/F_{k-1}\to0.
$$
and obtain the right exact sequence
$$
H_0(\oline\nf, V\otimes F_{k-1}) \to H_0(\oline\nf, V\otimes F_{k})\to H_0(\oline\nf, V)\otimes F_{k}/F_{k-1}\to0\, .$$
This implies
$$\Spec_{\Zc(\af+\mf)} H_0(\oline\nf, V\otimes F_{k})\subset \Spec_{\Zc(\af+\mf)} H_0(\oline\nf, V\otimes F_{k-1})\cup \Spec_{\Zc(\af+\mf)} (H_0(\oline\nf, V) \
\otimes F )$$
and the assertion follows by induction on $k$ and \eqref{Spec Hp3}.
\end{proof}

\begin{lemma}\label{lemma Lambda mu comp}
Let $\mu\in\cf_{\R}^{*}$ be dominant and integral and let $F_{\mu}$ be the highest weight representation with highest weight $\mu$. Let $\mu_{0}\in W_{\cf}\cdot \mu$ and let $\Lambda\in \R_{+}\mu_{0}$. Further, let $\nu\in [\Lambda]$, $\sigma\in\Spec_{\cf}F_{\mu}$ and  $w\in W_{\cf}$.
If $w(\Lambda+\mu_{0})=\nu+\sigma$, then $w\Lambda=\nu$ and $w\mu_{0}=\sigma$.
\end{lemma}

\begin{proof}
Let $r>0$ be so that $\mu_{0}=r\Lambda$.
We have $\sigma \in \Spec_\cf F_\mu \subset\operatorname{conv} (W_\cf\cdot \mu_0)$. In particular, $\|\sigma\|\leq\|\mu_{0}\|$. Moreover, $\|\nu\|=\|\Lambda\|$.
The Cauchy-Schwarz inequality applied to $\nu$ and $\sigma$ then gives that $\sigma=r\nu$.  It follows that $\sigma= w\mu_{0}$  and $\nu=w\Lambda$.
\end{proof}

For a Harish-Chandra module $U$ and infinitesimal character $[\Lambda_U]$ we define a
subset $[\Lambda_U]_\E\subset [\Lambda_U]$ by
$$
 [\Lambda_U]_\E:=\{ \Upsilon\in [\Lambda_U]\mid \Upsilon|_\af \in \E_U\}\, .
 $$

\begin{prop}\label{Prop U ds}
 For $U=W/J$ as defined after \eqref{def U} one has $[\Lambda_U]_\E\subset [\Lambda+\mu_0]^+$. In particular, $U$ is square integrable.
\end{prop}

\begin{proof} First recall that $W\subset V\otimes F_\mu$ is a direct summand as it is
a generalized $\Zc(\gf)$-eigenspace. Thus $H_0(\oline\nf,  W)\subset H_0(\oline \nf, V\otimes F_{\mu})$ as $MA$-module
and therefore
$$\Spec_{\Zc(\af+\mf)} H_0(\oline \nf, W) \subset - \rho +[\Lambda]^+ + \Spec_{\Zc(\af +\mf)} F_\mu$$
by Lemma \ref{lemma Spec HP}.  Now $U=W/J$ is a quotient of $W$ and thus the natural map
$H_0(\oline\nf, W)\twoheadrightarrow  H_0(\oline \nf, U)$ is surjective. We conclude that
\begin{equation} \label{incl U} \Spec_{\Zc(\af+\mf)} H_0(\oline \nf, U) \subset - \rho +[\Lambda]^+ + \Spec_{\Zc(\af +\mf)} F_\mu\, .\end{equation}
On the other hand we have $\Spec_{\Zc(\af+\mf)} H_0(\oline \nf, U) \subset - \rho +[\Lambda_U]$ by
Lemma  \ref{Lemma Hp}\eqref{spec Hp}. Comparing this with \eqref{incl U} and applying Lemma \ref{lemma Lambda mu comp} yields
$$\Spec_{\Zc(\af+\mf)} H_0(\oline \nf, U) \subset - \rho +[\Lambda_U]^+ \, .$$
Finally, from \eqref{L^2-condition} we deduce that $U$ is square integrable.
\end{proof}

Repeated application of Proposition \ref{Prop U ds} yields:

\begin{cor}\label{cor existence of ds on lines} There exists a $N\in\N$ such that if $V$ is a representation of the
discrete series with infinitesimal character $[\Lambda]$, then for every $k\in\N$ there exists a representation $U$ of the discrete series with infinitesimal character $[(kN+1)\Lambda]$ and  $\E_{U}\subset[ (kN+1)\Lambda]^{+}\big|_{\af}$.
\end{cor}

\begin{cor}\label{cor ds -> ds with strong reg inf char} Suppose that there exists a representation of the discrete series. Then there
exists a representation of the discrete series with strongly regular infinitesimal character.
\end{cor}

\begin{proof} Let $V$ be a representation of the discrete series with infinitesimal
character $[\Lambda]$ such that $\lambda=\Lambda\big|_{\af}\in \E_{V}$. By Corollary \ref{cor existence of ds on lines} there exists a discrete series representation  $V_{k}$ for every $k\in \N$  with infinitesimal character $[(kN+1)\Lambda]$ and $\E_k:=\E_{V_k} \subset [(kN+1)\Lambda]^{+}|_{\af}$. Since $[\Lambda]^{+}|_{\af}\subset -\operatorname{int} \Cc$, we have
$$
\lim_{k\to \infty} \operatorname{dist} (\E_k, -\partial \Cc )
\geq\lim_{k\to \infty}(kN+1) \operatorname{dist} ([\Lambda]^{+}|_{\af}, -\partial \Cc )
=\infty\, .
$$
It follows that for any $\mu \in \cf_\R^*$ there exists a $k$ such that
$$
\E_k + \operatorname{conv} \left(W_\cf \cdot \mu\big|_\af\right)
\subset -\operatorname{int}\Cc.
 $$
In view of  Lemma \ref{lemma tensor} this implies that  for every $m\in \N$ and any choice of fundamental representations $F_{\mu_1},  $ $\ldots, F_{\mu_m}$ there
exists a $n\in \N$ so that for every $k\in \N$ with $k\geq n$ all matrix coefficients of the representation
\begin{equation}\label{eq V_k tensor finite dim}
V_k\otimes F_{\mu_1} \otimes \ldots\otimes F_{\mu_m}
\end{equation}
are contained in $L^{2}(G)$. Let $\tilde{\Lambda}\in[\Lambda]$ be the dominant element with respect to $\Sigma_{\cf}^{+}$. In view of \cite[Theorem 1.2(1)]{Z} the representation (\ref{eq V_k tensor finite dim}) contains a subrepresentation with infinitesimal character
$[(kN+1)\tilde{\Lambda}+ \mu_1 + \ldots +\mu_m]$.

The proof will be finished by
showing that $(kN+1)\tilde{\Lambda}+ \mu_1 + \ldots +\mu_m$ is strongly regular for a suitable choice of
$\mu_{1},\dots, \mu_{m}$ and for all $k$ sufficiently large.
The strongly regular elements comprise the complement
of a finite union of proper subspaces of $\cf_\C^*$. We first choose $m$ and $\mu_1,\dots,\mu_m$ such that
$\mu:=\mu_1+\dots+\mu_m$ is outside of those subspaces which contain $\tilde\Lambda$.
Then so is $(kN+1)\tilde{\Lambda}+ \mu$ for any $k$.
Clearly each remaining subspace can contain
$(kN+1)\tilde{\Lambda}+ \mu$ for at most one value of $k$.
\end{proof}

\begin{cor}[Harish-Chandra]\label{cor HC-criterion} If a real reductive group $G$ admits a representation
of the discrete series, then there exists a compact Cartan subalgebra.
\end{cor}

\begin{proof} Combine Corollary \ref{cor ds -> ds with strong reg inf char} with Corollary \ref{cor sr implies elliptic}. \end{proof}


\begin{thebibliography} {10}

\bibitem{A} J. Adams, {\it Discrete series and characters of the component group}.
On the stabilization of the trace formula (eds. L. Clozel et al), Stabilization of the Trace Formula, Shimura Varieties, and Arithmetic Applications, vol. 1, pp. 369--387, Int. Press, Somerville, MA, 2011.

\bibitem{Casselman} W. Casselman, {\it Jacquet modules for real reductive groups}.
 Proceedings of the International Congress of Mathematicians (Helsinki, 1978), pp. 557--563, Acad. Sci. Fennica, Helsinki, 1980.

\bibitem{DKKS}  P. Delorme, F. Knop, B. Kr\"otz and H. Schlichtkrull, {\it Plancherel theory for real spherical spaces:
Construction of the Bernstein morphisms},  J. Amer. Math. Soc. {\bf 34} (2021), no. 3, 815--908.

\bibitem{HC1} Harish-{C}handra, {\it  Discrete series for semisimple Lie groups. I. Construction of invariant eigendistributions},  Acta Math. {\bf 113}, 241--318, 1965.

\bibitem{HC2}  \bysame, {\it Discrete series for semisimple Lie groups II},  Acta Math. {\bf 116} (1966), 1--111.

\bibitem{HS2} H. Hecht and W. Schmid ,  {\it Characters, asymptotics and $\nf$-homology of Harish-Chandra modules},
Acta Math. {\bf 151} (1983), no. {\bf 1--2}, 49--151.

\bibitem{K} A.W. Knapp, {\it Representation theory of semisimple groups. An overview based on examples.} Princeton Mathematical Series, 36. Princeton University Press, Princeton, NJ, 1986.

\bibitem{KKOS} B. Kr\"otz, J.J. Kuit, E.M. Opdam and H. Schlichtkrull,
{\it The infinitesimal characters of discrete series for real spherical spaces},
 Geom. Funct. Anal. {\bf 30} (2020), no. 3, 804–857.

\bibitem{Z} G. Zuckerman, {\it Tensor products of finite and infinite dimensional representations of semisimple Lie groups},
Ann. of Math. {\bf (2)} 106 (1977), no. {\bf 2}, 295--308.
\end{thebibliography}
\end{document}